\documentclass[10pt,a4paper]{article}
\usepackage{geometry}
\usepackage[english]{babel}
\usepackage{amsmath}
\usepackage{graphicx}
\usepackage{ marvosym }
\usepackage[utf8]{inputenc}
\usepackage{mathtools}
\usepackage{geometry}
\usepackage{amsfonts}
\usepackage{amssymb}
\usepackage{amsthm}
\usepackage{t1enc}
\usepackage[titles]{tocloft}
\usepackage{makeidx}
\usepackage{wasysym}
\usepackage{stmaryrd}
\usepackage{calc}  
\usepackage{enumitem} 
\usepackage[colorlinks=true,urlcolor=blue, linkcolor=blue,pageanchor=false, backref]{hyperref}
\usepackage{algpseudocode}
\usepackage{tikz}
\usetikzlibrary{arrows}
\usepackage{indentfirst}
\usepackage{xcolor}
\usepackage{ dsfont }
\usepackage{float}
\usepackage[abbrev,msc-links,backrefs]{amsrefs} 
\usepackage{doi}
\usepackage{todonotes}
\usepackage{soul}

\renewcommand{\PrintDOI}[1]{\doi{#1}}

\theoremstyle{plain}
\newtheorem{thm}{Theorem}[section]
\newtheorem{ques}[thm]{Question}
\newtheorem{claim}[thm]{Claim}

\newtheorem{fact}[thm]{Fact}
\newtheorem{prop}[thm]{Proposition}

\newtheorem*{prop*}{Proposition}
\newtheorem*{seged*}{Sublemma}
\newtheorem{cor}[thm]{Corollary}
\newtheorem{lem}[thm]{Lemma}

\newtheorem*{cond*}{Condition}
\newtheorem{conj}[thm]{Conjecture}

\newtheorem*{lem*}{Lemma}
\theoremstyle{definition}
\newtheorem{defn}[thm]{Definition}
\newtheorem*{defn*}{Definition}

\newtheorem{fel*}[thm]{Exercise}

\newtheorem*{megf*}{Observation}
\theoremstyle{remark}
\newtheorem{rem}[thm]{Remark}

\newtheorem{obs}[thm]{Observation}
\newtheorem*{rem*}{Remark}

\title{Reducing the dichromatic number via cycle reversions in infinite digraphs}
\author{Paul Ellis\thanks{ Department of Mathematics and Computer Science, Manhattanville College, 2900
Purchase Street, Purchase, NY 10577, USA\newline
 E-mail: \href{mailto:attila.: paulellis@paulellis.org}{: paulellis@paulellis.org} \newline
 Homepage: \url{http://paulellis.org/}
  } , Attila Joó\thanks{University of Hamburg and
  Alfréd Rényi Institute of Mathematics\newline
   E-mail: \href{mailto:attila.joo@uni-hamburg.de}{attila.joo@uni-hamburg.de} \newline
   Homepage: \url{https://www.math.uni-hamburg.de/home/joo/}
    } , Dániel T. Soukup\thanks{Universität Wien, Faculty of Mathematics, Kurt Gödel Research Center for Mathematical Logic, 
    Austria \newline
     E-mail: \href{mailto:: daniel.soukup@univie.ac.at}{daniel.soukup@univie.ac.at} \newline
     Homepage: \url{http://www.logic.univie.ac.at/˜soukupd73/}
      }}
\date{2019}
\begin{document}
\maketitle
\begin{abstract}
We prove the following conjecture of S. Thomassé: for every (potentially infinite) digraph $ D $ it is possible  to iteratively 
reverse 
directed cycles  in such a way that the dichromatic number of the 
final reorientation $ D^{*} $ of $ D $  is at most two and each edge is reversed only finitely many times.  In addition, we  
guarantee that in every strong component of $ D^{*} $ all the
local edge-connectivities are finite and any edge is reversed at most twice.
\end{abstract}

\section{Introduction}
Our general motivation is to analyse to what degree a complicated structure can be simplified by using only certain elementary 
operations. In our case, the 
structure is a digraph, the operation is reversing a directed cycle and we measure the complexity of the digraphs with the the 
dichromatic 
number.  The latter, introduced by V. Neumann-Lara  in 
\cite{dichro1st}, is a directed 
analogue of the chromatic number. 

\begin{defn}
The dichromatic number $ \boldsymbol{\chi(D)} $  of a digraph $ D $ is the smallest
cardinal $ \kappa $ such that $ 
V(D) $ can be coloured with $ \kappa $ many colours avoiding monochromatic directed cycles. 
\end{defn}

Various old and fascinating questions related 
to dichromatic number are still open.  For example, Neumann-Lara conjectured in \cite{NLara2col} that an orientation of a 
simple planar graph has always dichromatic number at most two. The best known partial result says that it is true whenever 
there is no directed cycle of length three (see \cite{planar}). In other type of problems the maximal dichromatic number of the 
possible orientations is in focus.  
Erd\H os and Neumann-Lara 
  conjectured the existence of a function $ f: \mathbb{N} \rightarrow \mathbb{N} $  such  that any graph $G$ with chromatic 
  number at 
least $ f(k)$ has an orientation $D$ with dichromatic number $ k $ (see \cite{Erd} and \cite{Erdcent}). We should emphasize 
that even 
the existence of $ f(3) $ is unknown and the analogous question 
considering infinite chromatic and dichromatic numbers is also open. Partial results were obtained, for example the independence 
of the statement ``Every 
graph of size and chromatic number $ \aleph_1 $ admits an orientation with dichromatic number $ \aleph_1 $''  (see 
\cite{cons}) 

C. Laflamme, N. 
Sauer and R. 
Woodrow drew the 
attention of the third author to the following conjecture of  
Thomassé:


\begin{conj}[S. Thomassé \cite{conj}]\label{conj}
For every (potentially infinite) digraph $ D $, it is possible  to reverse the edges of directed cycles iteratively in such a way that 
each edge is reversed only finitely many times and the dichromatic 
number of the 
final reorientation $ D^{*} $ of $ D $  is at most two.
\end{conj}

Thomassé et al. justified the conjecture for finite tournaments. In fact, they proved a stronger statement:
 
 \begin{thm}[Thomassé et al. \cite{gyarfas}]
 Let $ T $ be a tournament on the vertices $ v_1,\dots, v_n $ where for the outdegrees $ d^{+}(v_1)\geq d^{+}(v_2)\geq 
 \dots \geq d^{+}(v_n)  $ 
 holds. Then one can  reverse  directed cycles iteratively in such a way that  for the resulting tournament   $ 
 T^{*} $: each edge between the vertices $ 
 v_1, v_3, v_5,\dots $ goes forward as well as the edges between the vertices $ v_2, v_4, v_6 ,\dots$
 \end{thm} 

Note that  $ \chi(T^{*}) \leq2 $ holds in the theorem above. The restriction of the conjecture to arbitrary finite digraphs was 
solved by Charbit.  
\begin{thm}[P. Charbit, Theorem 4.5 of \cite{CharbitPhD}]\label{finite thm}
 In any finite digraph $ D $, it is always possible to reverse directed cycles  iteratively in such a way that $  \chi(D^{*})\leq 
 2$ 
 holds for the resulting orientation~$ D^{*} $.
 \end{thm}
The theorem is a consequence of the following interesting characterisation:
 
 \begin{thm}[P. Charbit, Theorem 4.4 of \cite{CharbitPhD}]\label{dich at most k}
Let $ D $ be a digraph and $ k\geq 1 $. Then  $ \chi(D)\leq k $ if and only if there exists a linear order $ < $ on $ V(D) $ such 
that for any 
directed cycle $ C $ of $ D $ at least $ \frac{\left|C\right|}{k} $ edges of $ C $ are  forward edges (with respect to $ < $).
 \end{thm}

 Indeed, for the given finite digraph $ D $ let us fix an arbitrary linear order $ < $ on $ V(D) $. If for each directed cycle at least  
 half of the edges 
 goes forward,  then $ \chi(D)\leq 2 $ by Theorem \ref{dich at most k}. Otherwise, we reverse a directed 
 cycle that violates this condition, which increases the total number of forward edges.

In the infinite case, the condition  ``each edge is reversed only finitely many times'' of Conjecture \ref{conj} has an important 
role. Without that we do not  have a well-defined orientation after  some limit step. Under this ``local finiteness'' assumption each 
edge has a 
stabilized orientation  before any limit step of a 
transfinite sequence of cycle reversions and hence 
 a natural limit 
 orientation can be defined. 
 
 Although Theorem \ref{dich at most k} remains true for infinite digraphs by compactness, the proof of Theorem \ref{finite thm} 
 based on it does not seem to adapt for the infinite 
 case. Indeed, after the reversion of a violating cycle the total number of forward edges may remain the same infinite cardinal.
 
 For the infinite case new ideas were needed. The  breakthrough  was due to 
  the first 
 and third author in \cite{ellis2019cycle} for infinite 
 tournaments. 
 They actually proved 
  more than  Conjecture \ref{conj}: one can transform any tournament by iteratively reversing  directed cycles to a linear order 
  ``modulo finite blocks''. More precisely:
 \begin{thm}[P. Ellis, D. T. Soukup,  \cite{ellis2019cycle}]\label{strong finite}
 In every tournament $ T $, one can iteratively reverse directed cycles  (reversing each edge only finitely often)  such that in 
 the resulting reorientation 
 $ T^{*} $ each strong 
 component is finite. 
 \end{thm}
 
  By applying
Theorem \ref{finite thm} to the (already finite) strong components separately, one can reduce the dichromatic number to at most 
two. 

Considering 
general 
 digraphs, one cannot hope to transform all the strong components to finite ones. Indeed, take $ \kappa \geq \aleph_0 $ many 
 directed cycles, pick one vertex in each and 
  identify these vertices.\footnote{This 
was pointed out by Carl Bürger  (personal communication).} The resulting digraph remains the same (up to isomorphism) after 
any 
iterative cycle reversion. 
\medskip

Our first main result is somewhat analogous to Theorem \ref{strong finite} but for arbitrary digraphs: we can reverse cycles to 
make the strong components have low connectivity.

\begin{thm}\label{local make small thm}
In every digraph $ D $, one can iteratively reverse directed cycles  (reversing each edge only finitely often)  such that in the 
resulting 
reorientation $ 
D^{*} $  
in each strong 
 component every local edge-connectivity is finite.
 \end{thm}

 Then, our second main result is the positive answer for the original conjecture in its whole generality:
 
 \begin{thm}\label{main result}
 In every digraph $ D $ one can reverse directed cycles iteratively (reversing each edge only finitely often)  such that $ 
 \chi(D^{*})\leq 2 $ holds for 
 the 
 resulting 
 reorientation $ D^{*} $.
 \end{thm}

The paper is organized as follows. We introduce some notation in the following section.  Then we develop the necessary tools in 
sections 3 until 6, from which 
we derive our main results in Section 7. The next section, Section 8, is devoted to  some structural consequences of the main 
results. We mention some 
open problems in Section 9. In the appendix, we collected the basic facts about elementary submodels that we use.

\subsection{Acknowledgements}

 The second author is grateful  for the support of  the Alexander von Humboldt Foundation and NKFIH 
 OTKA-129211. 
 
 The third author would like to thank the generous support of the FWF Grant I1921 and NKFIH OTKA-129211.
 
\section{Notation}

\subsection*{Graphs and digraphs} In graphs we allow multiple edges but not loops. More precisely, an \textbf{edge} $ e $ is 
an ordered pair $ \left\langle \{ u,v \}, i\right\rangle  $ 
where $ u\neq 
v $ are its end vertices  (the only role of index $ i $ is to represent multiple edges).  An \textbf{undirected 
graph} $ G $ is 
a set of 
edges, its 
vertex set $ \boldsymbol{V(G)} $ is the set of the end vertices of its edges.

The variable $ \boldsymbol{\overrightarrow{e}} 
$ stands for an orientation of an edge  $ e=\left\langle \{ u,v \}, i\right\rangle  $ of a graph $G$
from the two possibilities, more precisely $ \overrightarrow{e}\in \{  \left\langle u,v,i  \right\rangle, \left\langle v,u,i  \right\rangle 
\} $. Here $\left\langle u,v,i  \right\rangle $ stands for the orientation of $e$ from $u$ to $v$.
  A \textbf{digraph} $ D $ is what we obtain by 
orienting (the edges of) an undirected graph $ G $. The inverse operation $ \boldsymbol{\mathsf{Un}(D)}$ gives 
back the 
underlying 
undirected graph $ G $. We define $ \boldsymbol{V(D) }$  to be $ V(\mathsf{Un}(D)) $.
 The set of the outgoing edges of a vertex set $ W $ is denoted by $ \boldsymbol{\mathsf{out}_D(W) }$ and  $ 
 \boldsymbol{\mathsf{in}_D(W)} $ stands for the 
 ingoing edges. For $ E\subseteq \mathsf{Un}(D) $, we 
write 
$ \boldsymbol{D(E)} $ for the set of the $ D $-oriented elements of $ E $, i.e., the unique $ D'\subseteq D $ with $ 
\mathsf{Un}(D')=E $. A digraph $ D^{*} $ is a \textbf{reorientation} 
of $ D $ if 
$  \mathsf{Un}(D^{*})=\mathsf{Un}(D)  $. 

\medskip

 Cycles and paths are always meant to be directed. Formally, a \textbf{cycle} is a 
digraph of the form $ \{ \left\langle v_0,v_1, i_0 \right\rangle, \left\langle v_1,v_2,i_1 
\right\rangle, 
\dots\left\langle v_n,v_0,i_n\right\rangle    \} $ where $ n \geq 1$ and $ v_i $ are pairwise distinct. Paths are defined similarly.
The \textbf{local edge-connectivity} from $ u $ to $ v $ in a digraph $ D $ is the maximal size of a system  $ \mathcal{P} $ of 
pairwise edge-disjoint paths from $ u $ to $ v $ in 
$ D $ and it is denoted by $ \boldsymbol{\lambda(u,v;D)} $. The \textbf{strong components} of $ D $ are the equivalence 
classes 
of $ V(D) $ where $ u\sim v $ if $ \lambda(u,v;D), \lambda(v,u;D)>0 $. We refer to the subdigraphs spanned by the strong 
components also as strong 
components  but it will not lead to any confusion. A digraph is \textbf{strongly connected} (or \textbf{strong}) if it has only one 
strong 
component.

\subsection*{Set theory}
We use some standard set theoretic notation. We write $\boldsymbol{ \bigcup \mathcal{X}} $ for the union of the elements 
of the set family $ 
\mathcal{X} $. Variables $ \boldsymbol{\alpha, \beta, \gamma} $ and $\xi$ stand for ordinals and $ \boldsymbol{\kappa} $ 
denotes a cardinal. The set of the natural 
numbers is $ \boldsymbol{\omega} $. We say a set is \textbf{countable } if it is either finite or countably infinite. The restriction 
of a 
sequence $ s $ of length at least $ \alpha $ to $ \alpha $ is  $\boldsymbol{s \upharpoonright  \alpha} $. The concatenation of 
sequences $ s $ and $ z $ 
is denoted by 
$\boldsymbol{ s^\frown z} $. If a sequence has only one member $ C $, then we abuse the notation and write simply $ C $ 
for the sequence itself
as well. Finally, $\boldsymbol{ \left[X \right]^{\kappa}}  $ and $\boldsymbol{ \left[X \right]^{<\kappa}}  $ stand for the set of 
the $ \kappa $-sized and 
smaller than  $ 
\kappa $ subsets of $ X $ respectively.

\subsection*{Cycle reversions}

Let $ G $ be an  undirected graph and 
let $ 
\left\langle D_\xi:   \xi<\alpha\right\rangle $ be a sequence of orientations of $ G $. If the orientation of 
each  $ e\in G $ is stabilizing in the sense that there is some $ \xi_e<\alpha $ such that $ D_{\xi_e}(e)=D_{\xi}(e) $ whenever $ 
\xi \geq \xi_e $, 
then we define the 
\textbf{limit orientation}  $ D^{*} $ of the sequence by letting $ D^{*}(e):=D_{\xi_e}(e) $ for $ e\in G $.   It will be convenient  
to   always have 
some  kind of limit-type 
object even when the orientations of the edges are not stabilized. The 
set of 
\textbf{generalized 
limits} of  $ \left\langle D_\xi:   \xi<\alpha\right\rangle $ is defined to be the set of those orientations $ D^{*} $ of $ G $ for 
which for every finite $ 
F\subseteq G $ 
the set \[ \{ \xi<\alpha: D_\xi(F)=D^{*}(F) \} \] is unbounded in $ \alpha $. One can construct a generalized limit by taking an 
ultrafilter $ 
\mathcal{U} $ on 
$ 
\alpha $ containing all the non-empty terminal segments of $ \alpha $ and orienting an edge $ e\in G $ between $ u $ and $ v $ 
towards $ v $ iff
\[ \{ \xi<\alpha: D_\xi(e) \text{ points towards } v \}\in \mathcal{U}. \] 

 For a digraph $ D $, we define the set
  $ \mathbb{RS}(D)$  of \textbf{reversion sequences} acting on $ D $ and their 
  effect by the following recursion.  The 
  sequence  
$ 
\mathcal{C}=\left\langle 
C_\xi: \xi<\alpha  \right\rangle $ is in $  \mathbb{RS}(D)  $ and 
$ \boldsymbol{D \circlearrowleft \mathcal{C}}=D^{*}$ if 
\begin{itemize}
\item either $\alpha=0 $ and $ D^{*}=D $,
\item or $ \alpha=\beta+1 $,  $ (\mathcal{C} \upharpoonright\beta) \in 
\mathbb{RS}(D) $,  $ C_\beta $ is a directed cycle in $ D\circlearrowleft  (\mathcal{C}\upharpoonright \beta) $ and
we obtain $ D^{*} $ by reversing the direction of each edge in $ C_\beta $ in $ D\circlearrowleft  (\mathcal{C}\upharpoonright 
\beta) $,
\item or $ \alpha $ is limit ordinal,  $ (\mathcal{C} \upharpoonright\beta)\in 
\mathbb{RS}(D)  $ for $ \beta<\alpha $ and  
$ \left\langle D\circlearrowleft  (\mathcal{C}\upharpoonright \beta): \beta<\alpha   \right\rangle $ has a limit which is $ D^{*} $.
\end{itemize}

Note that if $ \mathcal{C}\in \mathbb{RS}(D)  $, then for every $ e\in \mathsf{Un}(D) $,  there can be only finitely many $ 
C_\beta $ that contain 
an orientation of $ e $.  We say that a $ 
\mathcal{C}\in \mathbb{RS}(D)  $ \textbf{touches} (uses) the edge $ e $ $ n $-times if 
there are exactly $ n $ cycles in the sequence   $ \mathcal{C} $ that contain an orientation of $ e $.  Also, we define $ 
\boldsymbol{E(\mathcal{C}) }$ to be $ \bigcup _{C\in 
\mathsf{ran}(\mathcal{C})}\mathsf{Un}(C) $.  

A reorientation $ D^{*} $ of $ 
 D $ is \textbf{reachable} from $ D $ if $ D^{*}=D\circlearrowleft 
 \mathcal{C} $ for a suitable $ \mathcal{C}\in \mathbb{RS}(D) $. We say that $ D^{*} $ is \textbf{locally reachable} from $ D 
 $ if 
 for all finite $ 
 F\subseteq \mathsf{Un}(D) $ there is a   $ \mathcal{C}\in \mathbb{RS}(D) $ such that 
 $ (D\circlearrowleft\mathcal{C})(F)=D^{*}(F) $.

\section{Reversion sequences and local reachability}

In this subsection we summarize some structural properties of reversion sequences. First of all, note that if $ \mathcal{E} $  is a 
rearrangement of some $ \mathcal{C}\in \mathbb{RS}(D) $ 
 in such a way that $ 
 \mathcal{E}\in \mathbb{RS}(D) $, then 
 necessarily $ D \circlearrowleft \mathcal{C}=D \circlearrowleft \mathcal{E} $. Indeed, the finial orientation of an $ e\in 
 \mathsf{Un}(D) $ depends 
 only on $ D(e) $ and on the parity of the number of cycles in the sequence containing some orientation of $ e $.
\begin{thm}[Theorem 8.4 of \cite{ellis2019cycle}]\label{DEomegaRev}
Suppose that $ \mathcal{C}\in \mathbb{RS}(D) $ is countable. Then there is a rearrangement $ \mathcal{E}\in \mathbb{RS}(D) 
$ of $\mathcal{C}$ of type $\leq \omega $.
\end{thm}

For every $ \mathcal{C}\in \mathbb{RS}(D) $, the edge sets of the members of $\mathcal{C}$ form a locally finite hypergraph 
whose connected components must be countable by König's Lemma. The subsequence of  $\mathcal{C}$ corresponding to a 
fixed component is clearly in $\mathbb{RS}(D)$. Moreover, these subsequences are pairwise edge-disjoint by construction. 
Applying Theorem \ref{DEomegaRev} to each of them and well-order the resulting sequences yields to the following extension 
of Theorem \ref{DEomegaRev}.
\begin{thm}\label{countable decomp}
 Let $ D $ be a digraph and let $ \mathcal{C}\in \mathbb{RS}(D) $ with $\left|\mathcal{C}\right|=:\kappa\geq \aleph_0 $. 
 Then there is a rearrangement $ \mathcal{E}\in\mathbb{RS}(D) $ of $ \mathcal{C} $ with length $ \kappa $.  Moreover, if 
 $ \kappa>\aleph_0 $  then one can choose $ \mathcal{E}$ in such a 
 way that 
 for $ \alpha<\kappa $, 
 \[ \mathcal{E}_\alpha:=\left\langle C_{\omega \alpha+n}:  n<\omega\right\rangle \in \mathbb{RS}(D), \] 
 and $ E(\mathcal{E}_\beta)\cap E(\mathcal{E}_\alpha)=\varnothing $ if $ \beta<\alpha<\kappa $ (where $ \omega \alpha$ 
 stands for ordinal 
 multiplication).
 \end{thm}
  By sorting the $ \omega $-intervals above according  if it uses  
 an edge 
 from a 
 given edge set $ E $ or not we obtain the following.
 \begin{cor}\label{small C}
Assume that $ D $ is a digraph, $ E\subseteq \mathsf{Un}(D) $ and $ \mathcal{C}\in \mathbb{RS}(D) $. Then there is a 
rearrangement 
$ \mathcal{E} $ of $\mathcal{C} $ such that $ \mathcal{E}={\mathcal{E}_E}^\frown   \mathcal{E}_{\neg E} $ where

\begin{itemize}
\item $\mathcal{E}_E,  \mathcal{E}_{\neg E}\in 
\mathbb{RS}(D)  $
\item $ E(\mathcal{E}_E)\cap E(\mathcal{E}_{\neg E})=\varnothing $,
\item $ E\cap E(\mathcal{E}_{\neg E})=\varnothing $,
\item the length of $  \mathcal{E}_E $ is $ \left|\mathcal{E}_E \right| \leq \left|E\right|+\aleph_0$.
\end{itemize}
Note that $ (D \circlearrowleft \mathcal{E})(E)= (D \circlearrowleft \mathcal{E}_E)(E) $ follows.
 \end{cor}

Now we turn our attention to properties related to local reachability.
 \begin{prop}\label{lim loc reach}
 Assume that $ \left\langle D_{\xi}: \xi<\alpha\right\rangle $ is a non-empty sequence of 
 orientations of a graph $ G $ and $D$ is an orientation of $G$ for which there is a $\xi_0<\alpha$ such that for all $\xi$ with 
 $\xi_0< \xi <\alpha $, $ D_\xi $ is locally reachable from $ D $.  Then
 every generalized limit of the sequence $ \left\langle D_{\xi}: \xi<\alpha\right\rangle $ is locally reachable from $ D $.
 \end{prop}
 \begin{proof}
 Let $ F \subseteq G$ be finite. By the definition of the generalized limit we can pick some $\xi$ with $ \xi_0<\xi<\alpha $ such 
 that 
 $ D_\xi(F)=D^{*}(F) $ and  $ D_\xi $ is locally reachable from $ D $.
 \end{proof}
  
 \begin{prop}\label{lok reach good orient}
 For every digraph $ D $, there is a reorientation $ D^{*} $ of $ D $ which is locally reachable from $ D $ and
   $ \chi(D^{*})\leq 2 $.
 \end{prop}
\begin{proof}
For finite $ D $, we simply apply Theorem \ref{finite thm}. For an infinite  $ D $, one can use standard compactness arguments. 
Indeed, take for 
example an ultrafilter $ \mathcal{U} $ on the finite subsets $ \left[D \right]^{<\aleph_0}  $ of $ D $ containing the set
$ \mathcal{X}_{\overrightarrow{e}}=\{ H\in\left[D \right]^{<\aleph_0} :  \overrightarrow{e}\in H \} $ for each $  
\overrightarrow{e}\in D $.
For every $ {H\in \left[D \right]^{<\aleph_0}}  $, we use Theorem \ref{finite thm} to fix a 
$ \mathcal{C}_H\in \mathbb{RS}(H) $ and a two-colouring $ c_H: V(H) 
\rightarrow 2$ such that $ c $ witnesses $ {\chi(H\circlearrowleft \mathcal{C}_H)\leq 2} $. It is easy to check that the unique 
orientation $ D^{*} $ 
satisfying
\[\{ H\in \mathcal{X}_{\overrightarrow{e}}:  D^{*}(e)=(H\circlearrowleft \mathcal{C}_H)(e)  \}\in \mathcal{U}\]
for every $ e\in \mathsf{Un}(D) $ is locally reachable from $ D $ and $ \chi(D^{*})\leq 2 $ witnessed by the unique $ c: 
V(D)\rightarrow 2 $ for which for every $ 
v\in V(D) $
\[\{ H\in \left[D \right]^{<\aleph_0}:v\in V(H) \wedge c_H(v)=c(v) \}\in \mathcal{U}.\]

\end{proof}

 \begin{prop}\label{fin many cycle}
  Let $ D $ be a digraph and let $ \mathcal{C}\in \mathbb{RS}(D) $ be finite. Then there is a $ \mathcal{E}\in 
  \mathbb{RS}(D)  $ which 
  consists of finitely many edge-disjoint cycles of $ D $ in some order such that 
  $ D\circlearrowleft \mathcal{E}=D\circlearrowleft \mathcal{C} $.
 \end{prop}
 \begin{proof}
We need to show that the set of edges  we  reverse  by applying $ 
\mathcal{C} $ to $ D $, namely set $ {R:= D\setminus(D\circlearrowleft \mathcal{C})} $, can be partitioned into cycles.  One 
can 
 show by induction on the length of $ \mathcal{C} $ that for each vertex $ v\in V(D) $, the number of the ingoing and 
 outgoing  edges of $ v $ in $ R $ are equal. Since $ R $ is finite, it follows  that the desired partition  can be constructed   
 ``greedily''. 
 \end{proof}
 \begin{prop}[Corollary 8.5 in \cite{ellis2019cycle}]\label{finite enough}
  Let $ D $ be a digraph, $ \mathcal{C}\in \mathbb{RS}(D) $, and let $ F\subseteq \mathsf{Un}(D) $ be 
  finite. Then there is a finite subsequence $ \mathcal{C}_F \in \mathbb{RS}(D)  $ of $ \mathcal{C} $ such that 
  $ (D\circlearrowleft \mathcal{C})(F)=(D\circlearrowleft \mathcal{C}_F)(F) $.
  \end{prop}
  \begin{cor}\label{loc reach trans}
 The local reachability relation is transitive, i.e., if $ D^{*}$ is locally reachable from $ D $ and $ D^{**} $ is locally reachable 
 from $ D^{*} $, 
 then 
 $ D^{**} $ is locally 
  reachable from 
  $ D $.
  \end{cor}
  \begin{proof}
  Let $ F\subseteq \mathsf{Un}(D) $ be finite. Pick a $ \mathcal{C}_F\in \mathbb{RS}(D^{*}) $ for $ D^{*} $ and $ D^{**} $ 
  as 
  in Proposition 
  \ref{finite enough}. 
  Apply Proposition \ref{finite enough} again, this time with $ E:=E(\mathcal{C}_F)\cup F, D $ and $ D^{*} $ to obtain $ 
  \mathcal{C}_E $. Then 
  $ (D \circlearrowleft \mathcal{C}_E \circlearrowleft \mathcal{C}_F)(F)=D^{**}(F)$.  
  \end{proof}
 
 Local reachability implies a formally stronger property, namely the countable version of itself.
 \begin{claim}\label{countable to everybody}
 If $ D^{*} $ is locally reachable from $ D $, then for all countable $ 
 E\subseteq \mathsf{Un}(D) $ there is a  $ \mathcal{C}_E\in \mathbb{RS}(D) $ of length at most $ \omega $ such that $ 
 \mathcal{C}_E $ 
 does not use any
 edge more than twice and $ {(D\circlearrowleft\mathcal{C}_E)(E)=D^{*}(E)} $. In particular,  among countable digraphs local 
 reachability implies 
 reachability.  
 \end{claim}
 \begin{proof}
 If $ E $ is finite, then by Propositions  \ref{fin many cycle} and \ref{finite enough}, we can choose 
 finitely many 
 edge-disjoint cycles in $ D $ in such a way that if $ \mathcal{C}_E $ is an enumeration of them then  $ 
 (D\circlearrowleft\mathcal{C}_E)(E)=D^{*}(E) 
 $. 
 
 Assume that $ E=\{ e_n \}_{n<\omega} $ and let $ \mathcal{C}_0:=\varnothing $. Suppose that $ \mathcal{C}_k\in 
 \mathbb{RS}(D) $ is already defined  for all $ k\leq n $ for some $n<\omega$ in such a way that 
 \begin{enumerate}
 \item  $ \mathcal{C}_n $ is finite,
 \item $ \mathcal{C}_j  \subseteq \mathcal{C}_k $ if $ j\leq k\leq n $,
 \item $ \mathcal{C}_n $ touches each edge at most twice,
 \item  if edge $ e $ is 
 touched by $ \mathcal{C}_n $ twice or $ e=e_k $ for some $ 
 k<n $ then $ (D \circlearrowleft \mathcal{C}_n)(e)=D^{*}(e) $.
 
 \end{enumerate}
 \begin{obs}\label{fin invertable}
  For a finite $ \mathcal{C}\in \mathbb{RS}(D) $, there is a (not necessarily unique) $ \mathcal{C}^{-1}\in 
  \mathbb{RS}(D\circlearrowleft \mathcal{C}) $ such that $ D\circlearrowleft \mathcal{C}\circlearrowleft 
  \mathcal{C}^{-1}=D $. Indeed, we can revere back the cycles starting from the last one.
  \end{obs}
 
 By  property 1 and Observation \ref{fin invertable}, $ D $ is reachable from   $ D \circlearrowleft \mathcal{C}_n $  and 
 hence $ D^{*} $ is locally reachable from $ D \circlearrowleft \mathcal{C}_n $ by Corollary \ref{loc reach trans}.  Let $ F:= 
 E(\mathcal{C}_n)\cup \{ e_k 
 \}_{k\leq n} $. According to the first paragraph of 
 this proof, there is a finite $ \mathcal{E}\in 
 \mathbb{RS}(D\circlearrowleft 
 \mathcal{C}_n)  $ consisting of edge-disjoint cycles for which we have
 $ (D\circlearrowleft \mathcal{C}_n \circlearrowleft \mathcal{E})(F)=D^{*}(F) $. We claim that $ 
 \mathcal{C}_{n+1}:={\mathcal{C}_n }^\frown \mathcal{E} $ maintains the conditions. Indeed, if $\mathcal{C}_n$ touches 
 some edge $e$ twice then $ (D \circlearrowleft \mathcal{C}_n)(e)=D^{*}(e) $ by induction. Since we have
 $ (D\circlearrowleft \mathcal{C}_n \circlearrowleft \mathcal{E})(e)=D^{*}(e) $, none of the cycles in $\mathcal{E}$ touches 
 $e$. Finally, if $\mathcal{C}_{n+1}$ uses an edge $f$ twice or $f\in \{ e_k 
 \}_{k\leq n} $, then $f\in F$ and  the choice of
 $\mathcal{E}$ ensures $ (D \circlearrowleft \mathcal{C}_{n+1})(f)=D^{*}(f) $. The recursion is done and   $ 
 \mathcal{C}_E:=\bigcup_{n<\omega}\mathcal{C}_n $ is as desired. 
 \end{proof}
\begin{rem}\label{count thm}
At this point we are able to prove the restriction of Theorem \ref{main result} to countable digraphs. Indeed, for a countable 
digraph $ D $ 
 pick a reorientation $ D^{*} $ with $ \chi(D^{*})\leq 2 $ which is locally reachable from $ D $  (see Proposition \ref{lok reach 
 good orient}) and 
 ``reach 
 it'' by Claim \ref{countable to everybody}. 
\end{rem}  

By combining Claim  \ref{countable to everybody} with Theorem 
\ref{countable decomp}, we obtain the following: 

\begin{cor}\label{touch at most 2}
For every $ \mathcal{C}\in \mathbb{RS}(D) $, there is a $ \mathcal{E}\in \mathbb{RS}(D) $ such that:
\begin{itemize}
\item $ D\circlearrowleft 
\mathcal{C}=D\circlearrowleft 
\mathcal{E} $,
\item the length of $\mathcal{E}$ is  $ \left| \mathcal{C}\right| $,
\item  $E(\mathcal{E}) \subseteq E(\mathcal{C}) $,
\item $ \mathcal{E} $ touches each edge at most twice.

\end{itemize}
\end{cor}
 
 \begin{rem}
 Claim  \ref{countable to everybody} is sharp in the sense that it may fail for uncountable $ E $. Indeed, consider the 
 digraph $D$ consisting of countably infinite $ u\rightarrow v $ edges and $ \aleph_1 $ many $ v \rightarrow u $ edges. The 
 reorientation in which each edge points towards $ v $ is locally reachable. Suppose for a contradiction that it is reachable as well 
 witnessed by $\mathcal{C}$. By Corollary \ref{touch at most 2}, we may assume that each edge is touched  at most twice by  
 $\mathcal{C}$. Then every $ v \rightarrow u $ edge is in exactly one cycle $C$ of $\mathcal{C}$. Moreover, no cycle in  
 $\mathcal{C}$ consists of two such edges because then one of them must be touched by another cycle as well. But then by the 
 Pigeonhole principle there must be an $ u\rightarrow v $ edge which is contained in uncountably many cycles of  $\mathcal{C}$ 
 which is a contradiction. A shorter argument can be given using Corollary \ref{loc con not incr}.
 \end{rem}

\begin{rem}
One cannot  replace  ``twice'' by ``once'' at the last point of Corollary \ref{touch at most 2}. On the one hand, it is easy to see that 
reversing pairwise 
edge-disjoint cycles cannot change the 
local edge-connectivities. On the other hand, a  reversion sequence can (see for example in Proposition \ref{EMeng flip sep}).
\end{rem}
\section{Local edge-connectivities and cycle reversions}
To prove Theorem \ref{main result} for arbitrary digraphs, knowing the theorem for countable ones (see Remark \ref{count 
thm}), it seems a natural idea to  partition $ D $ into  digraphs of smaller size by a chain of elementary submodels 
and solve the problem for 
these smaller pieces separately. (For readers not yet familiar with elementary submodels we suggest to look at our short 
overview in the Appendix \ref{appendix}.) 
This naive approach does not work in general  since we may have cycles not living in just a 
single member of the partition. One 
can 
make this approach work under the (very strong)  assumption that in every strong component of $ D $ every local 
edge-connectivity  is 
countable. Indeed,  in 
this case if $ M $ is a countable elementary submodel containing $ D $, then every cycle $ C $ of $ D $ that has an edge in 
$D\cap M $ is entirely 
in $ 
D\cap M $ (otherwise $ C\setminus M $ contains a $ u\rightarrow v$ path for some distinct vertices $ u,v\in M $  which ensures 
$ 
\lambda(u,v;D)>\left|M\right| $ by Fact \ref{big loc connect}). Furthermore, each  strong component of $ D\setminus M $ 
would contain at most one 
vertex 
from $ V(D)\cap M $ because of Fact \ref{big loc connect}. We would then be able to solve the problem for $ D\cap M $ and   
proceed 
with $ D\setminus M $ separately. Hence we shift our focus to 
investigate 
the 
possibility  of destroying uncountable 
edge-connectivities inside strong components by a reversion sequence.

Let us start with a famous result that we need later.
An \textbf{Erdős-Menger} $ \boldsymbol{u\rightarrow v }$ \textbf{path-system}  in a digraph $ D $ with distinct $ u, v \in V(D) 
$ is a set
$ \mathcal{P} $ of pairwise edge-disjoint $ u \rightarrow v $ paths such that one can choose exactly one edge from each $ P\in 
\mathcal{P} $ in 
such a way that the resulting edge set $ T $ is a $ u\rightarrow v $ cut in $ D $ (a $ u\rightarrow v $ cut is an edge set  such that 
every $ u\rightarrow v $ path of $ D $ must use an edge of it). 
We call such a $ T $ an \textbf{Erdős-Menger cut orthogonal to} $\boldsymbol{ \mathcal{P}} $. Note that $ 
\left|\mathcal{P}\right|= 
\lambda(u,v;D)$.
\begin{thm}[Infinite Menger's Theorem, \cite{aharoni2009menger}]\label{InfMeng}
For every digraph $ D $ and distinct $ u, v \in V(D) $, there is an Erdős-Menger $ u\rightarrow v $ path-system.
\end{thm}
\begin{rem}
Originally, Aharoni and Berger proved the vertex-version of Theorem \ref{InfMeng} but it is known to be equivalent with the 
edge-version above.
\end{rem}

\begin{claim}
For every digraph $ D $, $ \mathcal{C}\in \mathbb{RS}(D) $ and $ W\subseteq V(D) $, 
$ \left|\mathsf{out}_D(W)\right|\geq \left|\mathsf{out}_{D\circlearrowleft \mathcal{C}}(W)\right| $.
\end{claim}
\begin{proof}
Suppose to the contrary that  $ \left|\mathsf{out}_D(W)\right|< \left|\mathsf{out}_{D\circlearrowleft \mathcal{C}}(W)\right| 
$. If $  \left|\mathsf{out}_D(W)\right|<\aleph_0 $, then by Proposition \ref{finite enough}, a finite subsequence of $ 
\mathcal{C} $ would already increases the number of outgoing edges of $ W $ which is clearly impossible. Let $ 
\left|\mathsf{out}_D(W)\right| =:\kappa \geq \aleph_0 $. By taking a suitable initial segment of $ \mathcal{C} $, we may 
assume 
 $ \left|\mathsf{out}_{D\circlearrowleft \mathcal{C}}(W)\right|= \kappa^{+}$.  Without loss of generality we suppose
that $ \mathcal{C} $ has length $ \kappa^{+} $ (use Corollary \ref{small C} with $ 
E:=\mathsf{Un}(\mathsf{out}_{D\circlearrowleft \mathcal{C}}(W)) $). Note that for any $\alpha<\kappa^{+} $ for $ 
\mathcal{C}_\alpha:=\mathcal{C}\! \upharpoonright \alpha $, we have $ \left|\mathsf{out}_{D\circlearrowleft 
\mathcal{C}_\alpha}(W)\right|\leq\kappa $ since $ \mathcal{C}_{\alpha} $ reorients at most $ \kappa $ many edges. Let 
$ \beta_0=0 $. If $ 
\beta_n<\kappa^{+}$ is defined  then pick for each 
$ \overrightarrow{e}\in  \mathsf{out}_{D\circlearrowleft \mathcal{C}_{\beta_n}}(W) $ some $ \beta_e>\beta_n $ such that $ 
e\notin 
E(C_\beta) $ for $ \beta\geq \beta_e $. 
Let  $ \beta_{n+1}:=\sup\{\beta_e : e\in E(\mathsf{out}_{D\circlearrowleft \mathcal{C}_{\beta_n}}(W))  \} $ and finally $ 
\alpha:=\sup_{n<\omega}\beta_n<\kappa^{+} $. By construction,  none of the edges $ 
\mathsf{out}_{D\circlearrowleft 
\mathcal{C}_\alpha}(W) $ are 
touched by any $ C_\beta $ with $ \beta\geq \alpha$ and therefore none of the edges $ \mathsf{in}_{D\circlearrowleft 
\mathcal{C}_\alpha}(W) $ as well. Thus $ \mathsf{out}_{D\circlearrowleft 
\mathcal{C}}(W)=\mathsf{out}_{D\circlearrowleft \mathcal{C}_\alpha}(W) $ which is a contradiction.

\end{proof}
 \begin{cor}\label{loc con not incr}
 For every digraph $ D, \text{ distinct }u, v\in V(D) $  and $ \mathcal{C}\in \mathbb{RS}(D) $, 
 $ \lambda(u,v; D\circlearrowleft \mathcal{C})\leq\lambda(u,v; D)  $.
 \end{cor}
 \begin{proof}
 The local edge-connectivity from $ u $ to $ v $ can be expressed by Theorem \ref{InfMeng} as the smallest size of the $ 
 u\rightarrow v $ cuts. Thus by applying the previous claim we 
 obtain the following:
\begin{align*}
 \lambda(u,v; D) =&\min\{\left|\mathsf{out}_D(W) \right|: \{ u \}\subseteq W\subseteq V(D)-v \}\geq\\ 
 &\min\{\left|  \mathsf{out}_{D\circlearrowleft 
\mathcal{C}}(W)\right| :  \{ u \}\subseteq W\subseteq V(D)-v \}=\lambda(u,v; D\circlearrowleft \mathcal{C}). 
\end{align*}
 \end{proof}
\begin{prop}\label{path reset}
Let $ D $ be a digraph and let $ u, v \in V(D) $ with $u\neq v$ such that $ \lambda(v,u;D)\geq \aleph_0  $. Then for each $ u 
\rightarrow v $ path $ P $, there is a $ \mathcal{C}\in \mathbb{RS}(D) $ of length  $ \omega $ such that 
$ D\setminus (D\circlearrowleft \mathcal{C}) =P$ (i.e., $\mathcal{C}$ reverses exactly $ P $). 
In addition,  $ \mathcal{C} $ can be chosen in such a way that it avoids a 
prescribed  edge set of size less than $ \lambda(v,u;D) $ which is disjoint from $ P $.
\end{prop}
\begin{proof}
Let $ \{ Q_n \}_{n<\omega} $ be a set of pairwise edge-disjoint $ v \rightarrow u $ paths avoiding both $ P $ and a prescribed 
edge set of size  less 
than $ \lambda(v,u;D) $. 
Then $ P\cup Q_0 $ is the union of finitely many edge-disjoint cycles (see Figure \ref{fig pathflip}), reverse these cycles in an 
arbitrary order. 
In the 
resulting orientation the reverse of  $ Q_1 $ together with $ Q_2 $ is the union of finite many cycles thus we can 
reverse back $ Q_1 $ for the prize of reversing $ Q_2 $. By continuing this recursively we construct the desired 
$ \mathcal{C} $ which reverses exactly $ P $.
\begin{figure}[H]
\centering
\begin{tikzpicture}[scale=1.3]

\node[draw,circle] (v1) at (-2.5,0.5) {};
\node[draw,circle]  (v2) at (-1,0.5) {};
\node[draw,circle]  (v3) at (0.5,0.5) {};
\node[draw,circle]  (v4) at (2,0.5) {};
\draw[draw,circle]   (v1) edge[thick, ->] (v2);
\draw[draw,circle]   (v2) edge[thick, ->] (v3);
\draw[draw,circle]   (v3) edge[thick, ->] (v4);

\node at (2,0) {$v$};
\node at (-2.5,0) {$u$};
\node at (-0.5,0) {$P$};

\node[draw,circle]  (v5) at (1.5,1.5) {};
\node[draw,circle]  (v6) at (0,1.5) {};
\node[draw,circle]  (v7) at (-1.5,2) {};

\draw   (v4) edge[->] (v5);
\draw  (v5) edge[->]  (v3);
\draw   (v3) edge[->]  (v6);
\draw  (v6) edge[->]  (v7);
\draw  (v7) edge[->]  (v1);

\node[draw,circle]   (v8) at (1.5,2.5) {};
\node[draw,circle]   (v9) at (-0.5,2.5) {};
\node[draw,circle]   (v10) at (-2,2.5) {};
\draw  (v4) edge[dashed,->]  (v8);
\draw   (v8) edge[dashed,->]  (v9);
\draw  (v9) edge[dashed,->]  (v10);
\draw  (v10) edge[dashed,->]  (v1);

\node at (-0.6,1.2) {$Q_0$};
\node at (0,2.2) {$Q_1$};
\node at (0.2,3) {$\vdots$};
\end{tikzpicture}
\caption{Reversing path $ P $ without reversing anything else.}\label{fig pathflip}
\end{figure}
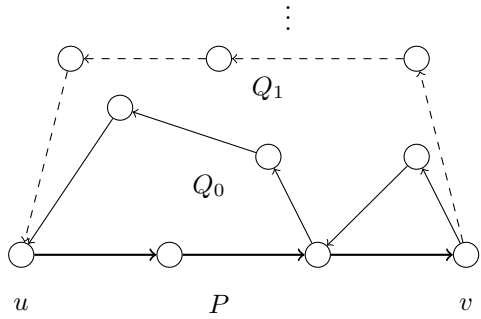 
\end{proof}
\begin{cor}\label{flip path-sys}
Let $ D $ be a digraph and $ u, v \in V(D) $ with $u\neq v$ such that $ \lambda(v,u;D)\geq \aleph_0  $.  Let $ \mathcal{P} $ be a 
set of 
edge-disjoint $ u\rightarrow v $ paths of size at most $ \lambda(v,u;D) $. Then there is a $ \mathcal{C}\in \mathbb{RS}(D) $ of 
length 
at most $ \lambda(v,u;D)  $ such that 
$ D\setminus (D\circlearrowleft \mathcal{C}) =\bigcup \mathcal{P}$. In addition,  $ \mathcal{C} 
$ can be chosen in such a way 
that it avoids a 
prescribed  edge set of size less than $ \lambda(v,u;D) $ which is disjoint from $ \bigcup\mathcal{P} $.
\end{cor}
\begin{proof}
Take a $ \left|\mathcal{P}\right| $-type enumeration of  $ \mathcal{P} $. We use transfinite recursion, in which we apply 
Proposition \ref{path reset} in each step with the next element of $ \mathcal{P} $ avoiding  the prescribed edge set and the edges 
we touched so far.
\end{proof}

\begin{prop}\label{EMeng flip sep}
If $ \mathcal{P} $  is an Erdős-Menger $ u\rightarrow v $ path-system  in a digraph $ D $ and $ D^{*} $ is the digraph that we 
obtain from $ D $ by 
reversing  the  edges from $ 
\bigcup \mathcal{P} $,  then $ D^{*} $ has no $ u\rightarrow v $ path.
\end{prop}
\begin{proof}
Let  $ T $ be an Erdős-Menger 
cut orthogonal to $ \mathcal{P} $  and we define $ W\subseteq V(D) $ to  be the set of those vertices that are 
reachable from 
$ u $ in $ D $ without using any edge from $ T $ (see Figure \ref{fig ErdMeng}). On 
the one hand,  $ \mathsf{out}_D(W)=T $. On the other hand, there is no $ P\in \mathcal{P} $ 
which 
uses an ingoing edge of $ W $ because such a $ P $ would meet at least two edges from $ T $ contradicting the fact that $ T $ is 
orthogonal to $ \mathcal{P} $. Thus by reversing $ \bigcup \mathcal{P} $ we reverse all the outgoing edges of $ W$ and 
none of 
the ingoing edges.
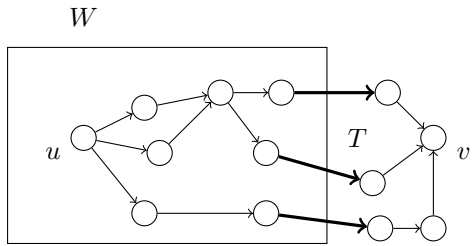
\begin{figure}[H]
\centering
\begin{tikzpicture}

\draw  (-1.8,0) rectangle (2.4,-2.6);

\node [draw,circle] (v1) at (-0.8,-1.2) {};
\node [draw,circle] (v2) at (0,-0.8) {};
\node [draw,circle] (v3) at (1,-0.6) {};
\node [draw,circle] (v4) at (1.8,-0.6) {};
\node [draw,circle] (v5) at (3.2,-0.6) {};
\node [draw,circle] (v6) at (3.8,-1.2) {};

\draw [ ->] (v1) edge (v2);
\draw [ ->] (v2) edge (v3);
\draw [ ->] (v3) edge (v4);
\draw [->, very thick] (v4) edge (v5);
\draw [ ->] (v5) edge (v6);
\node [draw,circle] (v7) at (0.2,-1.4) {};
\node [draw,circle] (v8) at (1.6,-1.4) {};
\node [draw,circle] (v9) at (3,-1.8) {};
\draw [->] (v1) edge (v7);
\draw [ ->] (v7) edge (v3);
\draw [ ->] (v3) edge (v8);
\draw [very thick, ->] (v8) edge (v9);
\draw [ ->] (v9) edge (v6);

\node [draw,circle] (v10) at (0,-2.2) {};
\node [draw,circle] (v11) at (1.6,-2.2) {};
\node [draw,circle] (v12) at (3.1,-2.4) {};
\node[draw,circle] (v13) at (3.8,-2.4) {};

\draw [ ->] (v1) edge (v10);
\draw [ ->] (v10) edge (v11);
\draw [very thick,->] (v11) edge (v12);

\node at (-0.8,0.4) {$W$};
\node at (-1.2,-1.4) {$u$};
\node at (4.2,-1.4) {$v$};
\node at (2.8,-1.2) {$T$};
\draw [ ->] (v12) edge (v13);
\draw [ ->] (v13) edge (v6);
\end{tikzpicture}
\caption{The Erdős-Menger 
cut $ T $ and the vertex set $ W $.}\label{fig ErdMeng}
\end{figure} 
\end{proof}

\begin{cor}\label{separate two}
If $ u\neq v $ belong to the same strong component of a digraph $ D $ and $ \lambda(u,v;D)+\lambda(v,u;D)\geq \aleph_0 $, 
then  there is a $ 
\mathcal{C}\in \mathbb{RS}(D) $ of length at most $ \lambda(u,v;D)+\lambda(v,u;D) $ such that $ u $ and $ v $ are in different 
strong component 
of $ D\circlearrowleft \mathcal{C} $. 
\end{cor}
\begin{proof}
By symmetry we may assume that $ \lambda(v,u;D)\geq  \lambda(u,v;D)$.  Take an Erdős-Menger $ u\rightarrow v $ 
path-system $ \mathcal{P} $ in 
$ D $  and reverse its edges applying Corollary \ref{flip path-sys}. By Proposition \ref{EMeng flip sep}, $ u $ and $ v $ are no 
longer in the same strong 
component and hence we are done.
\end{proof}

\section{Scattered digraphs and elementary submodels}
Instead of  just destroying all the infinite local edge-connectivities inside every strong component it will be more convenient to 
acquire a slightly 
stronger property by a suitable reversion sequence.
We call a digraph $ D $ \textbf{scattered} if  there is no strong component of $ D $ that can be subdivided by a suitable 
reversion 
sequence. More precisely,  
whenever $ u,v 
$ are 
in the same strong component of $ D $ they are in the same strong 
component of $ D \circlearrowleft \mathcal{C} $ for every 
$ \mathcal{C}\in \mathbb{RS}(D) $.  Corollary \ref{separate two} shows that being scattered implies that the local 
edge-connectivities are finite in every 
strong 
component.  A digraph $ D $ is $\boldsymbol{ (W,\kappa) }$\textbf{-scattered}, where $ W\subseteq V(D) $ and $ \kappa $ 
is a 
cardinal, if 
whenever $ u,v\in W $ are in the same strong component of $ D $ they are in the same 
strong 
component of $ D \circlearrowleft \mathcal{C} $ for every 
$ \mathcal{C}\in \mathbb{RS}(D) $ of length less than $ \kappa $. Note that  $ (V(D), \left|D\right|^{+}) $-scattered is 
equivalent with scattered.

\begin{prop}\label{max scat locons}
Let  $ D $ be $ (W,\kappa) $-scattered where $ \kappa >\aleph_0 $ and suppose that $ u, v\in 
W $ are distinct vertices in the same strong component of $ D $. If $ 
\lambda(u,v;D)+\lambda(v,u;D) $ is infinite, then $ \lambda(u,v;D), \lambda(v,u;D)\geq \kappa $.
\end{prop}
\begin{proof}
By symmetry we may assume that $ \lambda(u,v;D)\leq\lambda(v,u;D) $. Suppose for contradiction that $ 
\lambda(u,v;D)<\kappa $. Let $ 
\mathcal{P} $ be an Erdős-Menger
$ u\rightarrow v $ path-system. By applying Corollary \ref{flip path-sys}, there is a $ \mathcal{C}\in \mathbb{RS}(D) $ of 
length 
$ \lambda(u,v;D)<\kappa $  which reverses exactly $ \bigcup \mathcal{P} $ in $ D $. By Proposition
\ref{EMeng flip sep}, $ v $ is not reachable from $ u $ in $ D\circlearrowleft \mathcal{C} $.
Since $ u $ and $ v $ were in the same strong component of $ D $, it contradicts the fact  that $  D $ is $ (W,\kappa) $-scattered. 
\end{proof}

Our  goal is to prove that for every digraph $ D $  there is a  scattered $ D^{*} $  reachable from $ D $ (which is a strengthening 
of Theorem 
\ref{local make small thm}). Elementary submodels will play an important role in the proof, and we will need a couple of 
statements to be 
able to use them properly.

\begin{prop}\label{lock reach ins part}
Suppose that  $ M\supseteq \left|M\right| $ is an elementary submodel,  $ D\in M $ is a digraph and $ D^{*} $ (which may 
fail to be in $ M $) is locally reachable from $ D 
$. 
Then 
$ D^{*}\cap M $ is locally reachable from $ D\cap M $.
\end{prop}
\begin{proof}
Let $ F\subseteq \mathsf{Un}(D\cap M)$ be 
finite. By using Propositions \ref{finite enough} and \ref{fin many cycle}
and the fact that $ D^{*} $ is locally reachable from $ D $,  we conclude that there  is a 
finite sequence
$ \left\langle C_i\right\rangle_{i<n}=: \mathcal{C} \in \mathbb{RS}(D) $ of 
edge-disjoint cycles of $ D $ for which $ (D\circlearrowleft \mathcal{C})(F)=D^{*}(F) $. We replace those segments of the 
cycles that are not in $M$ with paths going inside $ M $ in the following way. For every  $ i<n $, the set $ C_i\setminus 
M$  consists of finitely many 
pairwise edge-disjoint paths, say $ P_{i,j} $ for $j<n_i$, where $ P_{i,j} $ goes from some $ u_{i,j} $ to some $ v_{i,j} $ with $ 
u_{i,j},v_{i,j}\in M 
$ (see Figure \ref{fig cycle inside}) and internally disjoint from $V(D)\cap M$.  
The existence of these paths 
implies (using Fact \ref{big loc connect}) that $ \lambda(u_{i,j},v_{i,j};D)>\left|M\right|$ and  $ \lambda(u_{i,j},v_{i,j};D\cap 
M)=\left|M\right| $. 
For $i<n$ and $j<n_i$, let $ Q_{i,j}$ be  pairwise edge-disjoint paths in $ D\cap M $ where $ Q_{i,j} $ is a $ u_{i,j} \rightarrow 
v_{i,j}$ path and they do not use any edge from the finite set  $ F \cup E(\mathcal{C}) $. For every vertex of the finite 
subdigraph  
 \[ \left( \bigcup_{i<n}C_i \setminus \bigcup_{i<n,j<n_i}P_{i,j}  \right)\cup  \bigcup_{i<n,j<n_i}Q_{i,j}, \] of $ D\cap M $, 
 the indegree is equal to the outdegree  therefore it is the union of edge-disjoint 
 cycles. By making a sequence $ \mathcal{E} \in \mathbb{RS}(D\cap M)$ from these cycles, 
 we have  
\[ ((D\cap M)\circlearrowleft \mathcal{E})(F)=(D\circlearrowleft \mathcal{E})(F)=(D\circlearrowleft 
\mathcal{C})(F)=D^{*}(F)=(D^{*}\cap M)(F). \] Thus $ D^{*}\cap M $ is  locally 
 reachable from $ D\cap  M$.
 \begin{figure}[H]
 \centering
\begin{tikzpicture}[scale=1.3]

\draw  (-1.8,0) rectangle (2.4,-2.6);

\node at (-1.4,0.2) {$M$};
\node [draw,circle] at (1.8,-0.6) {};
\node [draw,circle] (v1) at (1.8,-0.6) {};
\node [draw,circle] (v2) at (1.4,0.6) {};
\node [draw,circle] (v3) at (0.6,0.6) {};
\node [draw,circle] (v4) at (0.6,-0.4) {};
\node [draw,circle] (v5) at (0,-0.6) {};
\node [draw,circle] (v6) at (-0.4,0.6) {};
\node [draw,circle] (v7) at (-1,-0.6) {};
\node [draw,circle] (v8) at (-1,-1.4) {};
\node [draw,circle] (v9) at (0.2,-1.6) {};
\node [draw,circle] (v10) at (1.2,-1.6) {};

\draw [ dashed,->] (v1) edge (v2);
\draw [dashed, ->] (v2) edge (v3);
\draw [ dashed,->] (v3) edge (v4);
\draw [ ->] (v4) edge (v5);
\draw [ dashed,->] (v5) edge (v6);
\draw [ dashed,->] (v6) edge (v7);
\draw [ ->] (v7) edge (v8);
\draw [ ->] (v8) edge (v9);
\draw [ ->] (v9) edge (v10);
\draw [ ->] (v10) edge (v1);
\node at (-0.6,-1.8) {$C_0$};
\node at (1,1) {$P_{0,0}$};
\node at (-0.8,0.8) {$P_{0,1}$};
\node at (2,-1) {$u_{0,0}$};
\node at (1,-0.6) {$u_{0,0}$};
\node at (0,-1) {$u_{1,0}$};
\node at (-1.5,-0.6) {$u_{1,1}$};
\end{tikzpicture}
 \caption{Replacing the dashed part of cycle $ C_0 $.}\label{fig cycle inside}
 \end{figure}
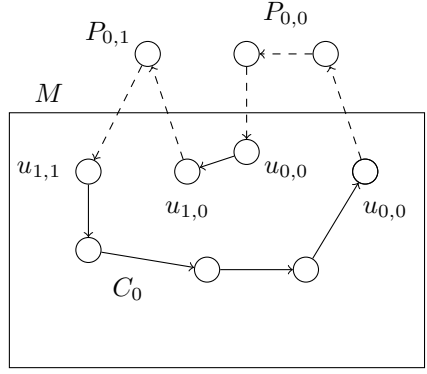
\end{proof}

Now we want to scatter a digraph $ D $ as much as possible, reversing edges only inside an elementary submodel $ M\ni D $ and 
remaining locally 
reachable from $ D $. In the applications we will  have 
to start with some $ L $ which may fail to be in $ M $ but which will be locally reachable from $ D $  and   identical to  $ D $ 
outside of $ M $.
\begin{lem}\label{loc reach scat down}
Suppose that  $ M\supseteq \left|M\right| $ is an elementary submodel,  $ D\in M $ is a digraph and  $ L $ (which may 
fail to 
be in $ M $) is locally reachable from $ D $  and satisfies
$ L\setminus M=D\setminus M $.  Then there is a $ (V(D)\cap M,\left|M\right|^{+}) $-scattered $ D^{*}=D^{*}(M,L) $ (does 
not depend on $ D $) 
with $ 
D^{*}\setminus M=D\setminus M $ 
which is locally 
reachable from $ D $. 

\end{lem}
\begin{proof}
Let $ \kappa:=\left|M\right| $ and let $ \left\langle \{ u_\alpha,v_\alpha \}: \alpha<\kappa^{+}  \right\rangle\ $ be a sequence 
with 
range 
$ \left[ V(D)\cap M\right]^{2}  $ in which the appearance of each element of  $ \left[ V(D)\cap M\right]^{2} $ is 
unbounded. We define a  sequence $ 
\left\langle 
L_\alpha: \alpha\leq \kappa^{+} \right\rangle  $  starting with  $ 
L_0:=L$. If $ L_\beta $ is $ (\{ u_\beta, v_\beta \},  \kappa^{+}) $-scattered, then let
$ L_{\beta+1} :=L_\beta  $. If it is not and $ \mathcal{C}_\beta $ is a witness for it, then let 
$ L_{\beta}' :=L_\beta\circlearrowleft \mathcal{C}_{\beta} $ and  we define $ L_{\beta+1} $ to be 
$ (L_{\beta}'\cap M) \cup (L\setminus M) $. If $ \alpha $ is a limit ordinal and $ L_\beta $ is defined
for $ \beta<\alpha $, then let $L_\alpha$ be a generalized limit of $ \left\langle L_{\beta}: \beta<\alpha\right\rangle $. Finally  we 
define $ D^{*} $ 
to be $ 
L_{\kappa^{+}} $. Note that  $ 
 L\setminus M=D\setminus M $ by assumption and the recursion preserves this property thus $ L_\alpha\setminus 
 M=D\setminus M $  for $ \alpha \leq \kappa^{+} $.

\begin{claim}\label{L alpha loc reach}
$ L_\alpha $ is locally reachable from  $ D $ for every $ \alpha\leq \kappa^{+} $.
\end{claim}
\begin{proof}
For $ \alpha=0 $ it holds by assumption. For limit steps, it follows from Corollary \ref{lim loc reach}. For a successor step, $ 
L_{\beta}'=L_\beta\circlearrowleft \mathcal{C}_{\beta} $ is locally reachable from 
$ D $ since $ L_\beta $ is locally reachable by induction. By applying Proposition \ref{lock reach ins part} with $ 
M,D$ and $ L_{\beta}' $, 
we obtain that $ L_{\beta+1}\cap M  $ is locally reachable from $ D\cap M $. Since 
$  L_{\beta+1}\setminus M=D\setminus M $, it implies that 
$ L_{\beta+1} $ is locally reachable from~$ D $.
\end{proof}
\begin{claim}\label{sep 4ever}
If $ L_{\beta+1}\neq L_\beta $ for some $ \beta<\kappa^{+} $, then $ u_\beta $ and $ v_\beta $ are in different 
strong components of $ L_\alpha $ for $ \beta<\alpha \leq \kappa^{+} $.
\end{claim}
\begin{proof}
By construction, $ u_\beta, v_\beta $ are in different strong components of $ L_{\beta}' $. By symmetry we may  assume 
that $ v_\beta $ is 
not reachable from $ u_\beta $ in $ L_{\beta}' $. First we show that this remains 
true in $ L_{\beta+1} $ as well. Suppose for a contradiction that $ P $ is a $ u_\beta\rightarrow v_\beta $ path in $ 
L_{\beta+1} $. We show that we must have such a path already in $ L_{\beta}' $. Since $ L_{\beta+1}\cap M = 
L_{\beta}'\cap M $, we have $ M\cap P\subseteq L_\beta' $. 
We prove that the  $ P\setminus M $ part of $ P $  can be replaced in $ L_{\beta}' $. 
  Indeed, $ P\setminus M $ consists of edge-disjoint paths $ Q_i$ for $i<n$ where $ Q_i $ is a $ u'_i\rightarrow v'_i $ path
  for some distinct $ u'_i, v'_i\in V(D)\cap M $ and internally disjoint from $V(D)\cap M$. Since  $ L_{\beta+1}\setminus 
  M=D\setminus M $, each $ Q_i $ lies in $ D\setminus M $ and then Fact \ref{big loc 
  connect} shows 
  that
  $ \lambda(u_i,v_i; D)>\left|M\right|=\kappa$. Because of $ \left|D\vartriangle 
  L_{\beta}'\right|\leq \left|M\right|+\left|\mathcal{C}_\beta\right|\cdot \aleph_0=\kappa $, we may conclude that  $ 
  \lambda(u_i,v_i; L_{\beta}')>\kappa $. In particular, $ v_i $ is reachable 
  from $ u_i $ in $ L_{\beta}'$, but then $ 
  v_\beta $ is reachable from $ u_\beta $ in $ L_{\beta}' $ which is a contradiction.

Let $ T \subseteq L_{\beta+1}$ be a directed cut  witnessing the non-existence of a path from  $ u_\beta $ to $ v_\beta $ 
in  $ L_{\beta+1} $. Note that the
edges 
$ \mathsf{Un}(T\setminus M) $ are oriented in $ T $ as in $ D $ because $(T\setminus M)\subseteq (L_{\beta+1}\setminus 
M)=(D\setminus 
M) $. We show by transfinite induction that $T\subseteq L_\alpha $  whenever $ \alpha \geq \beta+1 $.  The initial step and 
the limit steps are obvious. Suppose we know the statement for some $ \alpha $. Since the orientation  $ L_{\alpha}' $ is 
reachable from $ 
L_\alpha $ by the construction, $ T\subseteq  L_{\alpha}'$ also holds. We obtain $ L_{\alpha+1} $ from  $ L_{\alpha}' $ by 
reversing the 
edges outside of $ M $ to the direction defined by $ D $. Since $ D $ and $ T $ orients the edges $ \mathsf{Un}(T\setminus M) 
$ the same way, it does not reverse any edge in $ T $ and hence $ T\subseteq L_{\alpha+1} $.
\end{proof}
 It remains to show that $ D^{*} $ is $ (V(D)\cap M,\kappa^{+})$-scattered. For every distinct 
$ u, v\in V(D)\cap M $, if there is an ordinal $ \beta $ for which  
$ \{ u_\beta, v_\beta \}=\{ u,v \} $  and $ L_\beta $ 
is not $ (\{ u_\beta, v_\beta \}, \kappa^{+}) $-scattered then  it must be unique by Claim 
\ref{sep 4ever}  and we define $ \beta_{u,v}:=\beta+1 $. 
If 
there is no 
such a $ \beta $, then $ \beta_{u,v}:=0 $. The terminal 
segment of the sequence $ \left\langle L_\alpha: \alpha<\kappa^{+}\right\rangle  $ from $ \sup \{\beta_{u,v}:  \{ u,v \}\in 
\left[ V(D)\cap M\right]^{2} 
\}<\kappa^{+}$ is constant. Since the appearance of each  $ \{ u,v \}\in \left[ V(D)\cap M\right]^{2} $ is unbounded in
$ \left\langle \{ u_\alpha,v_\alpha \}: \alpha<\kappa^{+}  \right\rangle\ $, it ensures that
 $ L_{\kappa^{+}} $ is $ (V(D)\cap M, \kappa^{+})$-scattered.
\end{proof}

The last tool we will  need about elementary submodels is an easy technical lemma which is a strengthening of Fact \ref{exist 
submod}. One can show that every uncountable elementary submodel can be written as the union of an increasing 
continuous chain of 
smaller sized elementary submodels. We will need a stronger  property which is not automatically true.  We say that 
the elementary submodel $ M $ has the  \textbf{chain property}  if either $ \left|M\right|=\aleph_0 $ or $ 
\left|M\right|=\kappa>\aleph_0 $ and there is a sequence $ \left\langle M_\alpha: \alpha<\kappa\right\rangle  $  of elementary 
submodels such that for all $\alpha<\kappa$:
\begin{enumerate}
\item $ \alpha\in M_{\alpha+1} $,
\item $ \left|\alpha\right|\leq \left|M_\alpha\right|<\left|M\right| $,
\item $ \left\langle M_\beta: \beta\leq\alpha \right\rangle \in M_{\alpha+1} $,
\item $M_\alpha=\bigcup_{\beta<\alpha}M_\beta$ if $\alpha$ is a limit ordinal,
\item $ M_\alpha $ has the chain property,
\item $ M=\bigcup_{\alpha<\kappa}M_\alpha $.
\end{enumerate}

\begin{obs}\label{Obs chain}
\begin{itemize}\ 
\item For an  elementary submodel $ M $ with the chain 
property  $ \left|M\right|\subseteq M $ holds.
\item  $ M_\beta\cup \{ M_\beta \} \subseteq M_\alpha $ for  all $ \beta<\alpha $ in the chain above. Indeed,  $ \beta, \left\langle 
M_\gamma: \gamma\leq \beta \right\rangle\in   M_{\beta+1} $ by assumption and  $M_\beta$ is definable from them thus 
$M_\beta \in M_{\beta+1}$. Then by applying 
Fact 
\ref{contains elem}, $ M_\beta \subseteq M_{\beta+1} $.
\item The chain property is well-defined since to decide if $ M $ has it, it is enough to know 
this  for elementary submodels which are elements of $M$.
\item If $ \left\langle M_\alpha: \alpha<\kappa\right\rangle  $ witnesses the chain property of $ M $, then so does every 
non-empty terminal segment of it.

\end{itemize}
\end{obs}
 
\begin{prop}\label{exist chain prop}
For every infinite cardinal $ \kappa $ and set $ X $, there is an  elementary submodel $ M\ni X $ of size $ \kappa $ that has 
the chain property.
\end{prop}
\begin{proof}
We use induction on $ \kappa $. For $ \kappa=\aleph_0 $, it follows from Fact \ref{exist submod}. Suppose that $ 
\kappa>\aleph_0 $ and the statement is known for all cardinals smaller than $ \kappa $. We build the sequence in the definition  
by 
transfinite recursion and define $ M $ as the union of the chain. Let $ M_0 $ be an arbitrary countable elementary submodel 
containing $ X $.
To build $ M_{\alpha+1} $, we use the induction hypothesis with $ \left|\alpha\right|+\aleph_0 $ and  $X':=\{\alpha, 
\left\langle M_\beta: \beta<\alpha \right\rangle  \}   $. 
\end{proof}

\section{The key lemma}\label{sec key lemma}
The main  difficulty of the elementary submodel approach for this problem is the following. After applying a reversion sequence  
inside the part of the 
digraph $ D $ covered by some elementary submodel 
 $M $ it is advisable  to avoid the edges in $ M $ in further cycle reversions (otherwise we may eventually reverse some edge 
 infinitely often). How can 
 we guarantee that this is possible? We introduce a lemma that helps us to overcome this difficulty.  Whenever we have $ 
 H\subseteq D $ with 
 some special 
properties, the lemma guarantees that for any  $ \mathcal{C}\in \mathbb{RS}(D)  $   we can find an $ 
\mathcal{E}\in \mathbb{RS}(D\setminus H)  $ that reverses exactly the same edges of  $ D\setminus H $ as $ 
\mathcal{C} $.
\begin{lem}\label{avoid down before}
Let $ D $ be a digraph and let $ H\subseteq D $ be infinite. Assume $ {\lambda(u,v;D), \lambda(v,u;D)>\left|H\right|} $  for 
every distinct $ u, v\in 
V(H) $ which are in the same 
strong component 
of $ D $  and satisfy $ {\lambda(u,v;D\setminus H)>0} $. Then for 
every $ 
\mathcal{C}\in \mathbb{RS}(D)  $, there is a
$ \mathcal{E}\in \mathbb{RS}(D\setminus H)  $ such that for every $ e\in \mathsf{Un}(D\setminus H) $ we have
$ ((D\setminus H)\circlearrowleft \mathcal{E})(e)=(D\circlearrowleft \mathcal{C})(e) $.
\end{lem}
\begin{proof}
We may assume that the length of $ \mathcal{C} $ is a (possibly finite) cardinal 
$\kappa \leq \left|H\right|+\aleph_0 $ (use Corollary \ref{small C} with $ E:=\mathsf{Un}(H) $ and deal only with the part $ 
\mathcal{E}_E $). Let
$ \mathcal{C}=\left\langle C_\xi: \xi<\kappa  \right\rangle $ and let us write $ \mathcal{C}_\alpha $ for 
$ \mathcal{C} \upharpoonright \alpha$.  We construct by transfinite recursion an increasing continuous sequence 
$ \left\langle \mathcal{E}_\alpha: \alpha\leq \kappa  \right\rangle  $ where 
$ \mathcal{E}_{\alpha}\in \mathbb{RS}(D\setminus H)  $ such that for each $ e\in D\setminus H $, the reversion sequence $ 
\mathcal{E}_\alpha $ 
touches $ e $
either the same number of times as $ \mathcal{C}_\alpha $ or, for at most $ \left|\alpha\right|+\aleph_0 $ many $ e $, exactly two 
times 
more. Obviously $ \mathcal{E}_0:=\varnothing $ is appropriate. Suppose that $ \mathcal{E}_\alpha $ is defined and the 
conditions hold so far. Note that for $ e\in \mathsf{Un}(D\setminus H) $ 
we 
have
$ ((D\setminus H)\circlearrowleft \mathcal{E}_\alpha)(e)=(D\circlearrowleft \mathcal{C}_\alpha) (e)$. 

If $ C_\alpha \subseteq 
H $, let 
$ \mathcal{E}_{\alpha+1}:=\mathcal{E}_\alpha $. If $ C_\alpha\cap H=\varnothing $, then we define  
 $ \mathcal{E}_{\alpha+1}$ to be ${\mathcal{E}_\alpha}^\frown C_\alpha $. The remaining case is when $ C_\alpha $ has 
 edges 
 both
 in $ H $ and in  $ D\setminus H $. Then $ C_\alpha \setminus H $ consists of 
 finitely many 
pairwise edge-disjoint paths, say $ P_i$ for $ i<n $, where $ P_i $ goes  from  $ u_{i} $ to some $ v_{i} $ for some $ 
u_{i},v_{i}\in V(H) 
$ and has no internal vertex in $V(H)$. The 
$ u_{i},v_{i} $ are 
in the same strong component in $ D\circlearrowleft 
\mathcal{C}_\alpha $ and hence  in $ D $ as well (see Corollary 
\ref{loc con not incr}). Path $ P_i $ witnesses the fact that
$ {\lambda(u_i,v_i; (D\setminus H)\circlearrowleft \mathcal{E}_\alpha)>0} $ but then by Corollary 
\ref{loc con not incr}, $ \lambda(u_i,v_i; D\setminus H)>0 $. By the assumption about $ H $, we have $ 
\lambda(u_i,v_i;D),\lambda(v_i,u_i;D)>\left|H\right| $. So for each $ P_i $ pick a $ v_{i}\rightarrow u_{i} $ path $ Q_i $ such 
that $ Q_i $ has no 
common edges 
with:
$ H, \mathcal{E}_{\alpha},   Q_j$ for $ j\neq i$ and $  P_j$ for $j<n $. One can partition $ \bigcup_{i<n}P_i\cup Q_i $ into 
finitely many 
edge-disjoint cycles. Extend $ \mathcal{E}_\alpha $ with these cycles in an arbitrary order and then apply Proposition 
\ref{path reset} (combined with Corollary \ref{touch at most 2}) to reverse back the paths $ Q_i $ without reversing any other 
edges and without using 
any edge that we touched 
already. Let $ \mathcal{E}_{\alpha+1} $ be the resulting extension. Limit steps 
preserve the conditions automatically and $ 
\mathcal{E}:=\mathcal{E}_\kappa $ is suitable. 
\end{proof}
\begin{cor}\label{avoid down}
Assume that $ M $ is an elementary submodel, $ D\in M $ is a digraph and $ L $ (which may fail to be in $ M $) is a $ 
(V(D)\cap M, \left|M\right|^{+}) $-scattered reorientation of $ 
D $ satisfying $ L\setminus M=D\setminus M $. Then for 
every $ 
\mathcal{C}\in \mathbb{RS}(L)  $, there is a
$ \mathcal{E}\in \mathbb{RS}(D\setminus M)  $ such that $ ((L\setminus M)\circlearrowleft 
\mathcal{E})(e)=(L\circlearrowleft \mathcal{C})(e) $ 
for $ e\in \mathsf{Un}(D\setminus M) $.
\end{cor}
\begin{proof}
It is enough to show that $ L $ and $ H:=L\cap M $   satisfy the premises of  Lemma \ref{avoid down before}.  Suppose that 
$ 
u,v\in V(L)\cap M$ are distinct vertices from the 
same strong 
component of $ 
L $ 
and 
assume that $ \lambda(u,v;L\setminus M)>0 $. Since $ L\setminus M=D\setminus M $, 
it is equivalent with $ \lambda(u,v;D\setminus M)>0 $. By assumption $ D\in M $, therefore  by Fact \ref{big loc connect}, $ 
\lambda(u,v;D)>\left|M\right| $ and hence $ 
\lambda(u,v;L)>\left|M\right| $ because $ D $ and $ L $ orient at most $ \left|M\right| $ edges differently. Since $ L $ is $ 
(V(L)\cap M, 
\left|M\right|^{+}) 
$-scattered, Proposition 
\ref{max scat locons} ensures that  $ 
 \lambda(v,u;L)>\left|M\right| $ as well. 
\end{proof}
\section{Proof of the main results}

Let us restate the main results for convenience.
\begin{thm}\label{max scat from everybody}
For  every digraph $ D $,  there is a $ \mathcal{C}\in   \mathbb{RS}(D)$ such that 
$D\circlearrowleft \mathcal{C} $ is  scattered and (hence)  in every strong component of $D\circlearrowleft 
\mathcal{C} $ all the local edge-connectivities are finite.
\end{thm}
\begin{thm}\label{all dichrom}
For every digraph $ D $, there is a $ \mathcal{C}\in   \mathbb{RS}(D)$ such that 
$ \chi(D\circlearrowleft \mathcal{C})\leq 2 $.
\end{thm}

In order to prove Theorems \ref{max scat from everybody} and \ref{all dichrom}, it is enough to show the following lemma.
\begin{lem}\label{finalLemma}
For every elementary submodel $ M $ that has the chain property \footnote{see the definition right before Section \ref{sec key 
lemma}}, for every digraph $ D\in M $ and for every 
reorientation $ R $ of $ D $ (which may fail to be in $ M $)
there are  $ {\mathcal{C}_{M,D}, \mathcal{C}_{M,D, R}\in \mathbb{RS}(D\cap M)}$,  where $ \mathcal{C}_{M,D} $ does 
not depend on $ R $ and they satisfy the following 
properties:
\begin{enumerate}

\item  $D\circlearrowleft \mathcal{C}_{M,D} $ is $ (V(D)\cap M, \left|M\right|^{+}) $-scattered,
\item If $ R $ is  locally reachable from  $ D\circlearrowleft \mathcal{C}_{M,D} $, then  
$(D\circlearrowleft \mathcal{C}_{M,D,R})\cap M=R\cap M  $. 
\end{enumerate}
\end{lem}

Let us show first  that Lemma \ref{finalLemma} implies Theorems \ref{max scat from everybody} and \ref{all dichrom}. Let a 
digraph $ D $  be given.  By 
Proposition \ref{exist chain prop}, we 
can pick an elementary submodel $ M\ni D $ of size $\left|D\right| $ with the chain property. Note that $\left|D\right|= 
\left|M\right|\subseteq M 
$ and hence $ D\subseteq M $ by Fact \ref{contains elem}.  Then   
  $D\circlearrowleft \mathcal{C}_{M,D} $ is scattered  by property 1 which (together with Corollary \ref{separate two}) 
proves Theorem 
\ref{max scat from everybody}. Then by choosing 
an $ R $ with
$ \chi(R)\leq 2 $ which is locally reachable from $ D\circlearrowleft 
\mathcal{C}_{M,D} $ (which exist due to Proposition \ref{lok reach good orient}), $ \mathcal{C}_{M,D,R} $ satisfies the 
requirement of 
Theorem \ref{all dichrom}.

\begin{proof}[Proof of Lemma \ref{finalLemma}]
 To show that the desired $ \mathcal{C}_{M,D}, \mathcal{C}_{M,D, R}\in \mathbb{RS}(D\cap M)$ exist, we apply 
 transfinite 
 induction on $ \kappa:=\left|M\right| $. Suppose 
 that $ \kappa=\aleph_0  $.
First we pick a reorientation $ D^{*}$ of $ D $ with $ D^{*}\setminus M=D\setminus M $ which is locally reachable from $ D $ 
and $ (V(D)\cap M, \aleph_1) $-scattered (use  Lemma 
\ref{loc reach scat down} with $ L:=D $). By Proposition \ref{lock reach ins part},  $ D^{*}\cap M $ is  locally reachable from 
$ 
D\cap  M$. Since  $ D\cap  M $ is countable, local reachability implies reachability 
by Proposition 
\ref{countable to everybody}.  Let $ \mathcal{C}_{D,M} $ be a witness for this reachability.  Suppose that $ R $ is  locally 
reachable from 
$ D^{*} $. By the transitivity of local reachability (see Proposition \ref{loc reach trans}), $ R $ is locally reachable from $ D 
$.  We apply Proposition \ref{lock reach 
ins part} again to conclude that
 $ R\cap M $ is locally reachable from $ D\cap M $ and hence by 
Proposition 
\ref{countable to everybody} reachable too. We choose $ \mathcal{C}_{D,M,R} $ to exemplify this reachability.

Let $\kappa>\aleph_0 $.  Since  $ M $ has the chain 
property by assumption, 
there is  an 
increasing continuous 
 chain $ \left\langle M_\alpha: 1\leq\alpha<\kappa
\right\rangle $   of elementary submodels such that all the members have the chain property, $ \alpha\in M_\alpha $, 
 $\left|\alpha\right|\leq \left|M_\alpha\right|<\kappa $, 
 $ \left\langle M_\beta: \beta\leq\alpha  \right\rangle\in M_{\alpha+1}  $ for $ \alpha<\kappa $ and $ 
 \bigcup_{\alpha<\kappa}M_\alpha=M=:M_\kappa $. 
 Recall that Fact \ref{contains elem}  ensures 
 $ \alpha 
 \subseteq M_\alpha $ and 
 $ M_\beta\in M_\alpha $ for $ \beta<\alpha $. 
 Since $ M_\kappa $ contains $ D $ by assumption so does some $ M_\gamma $ 
 where $ 
 1\leq\gamma<\kappa $.  By taking a terminal segment of the sequence,  we may assume that  $ \gamma=1 $ (see  
 Observation \ref{Obs chain}). For technical reasons let us define  $ M_0:=\varnothing $.  We construct a  sequence 
 $ \left\langle D_\alpha: \alpha\leq \kappa \right\rangle  $ of reorientations of $ D $ 
   by transfinite recursion such that for every $ \alpha\leq \kappa $:
  \begin{enumerate}
  \item[i.] $ D_0=D $,
  \item[ii.] $ D_\beta\in M_{\alpha} $ for $ \beta<\alpha $,
  \item[iii.] $ D_\alpha $ is $ (V(D)\cap M_\alpha, \left|M_\alpha\right|^{+}) $-scattered,
\item[iv.] $ D_\alpha\setminus M_\alpha=D \setminus M_\alpha $,
\item[v.]  $ D_\alpha$  is locally reachable from   $ D_\beta $ for $ \beta<\alpha $.

  \end{enumerate}

 Suppose that $ D_\beta $ is defined for all $ \beta<\alpha $ and the 
conditions i-v hold so far. For $ \alpha=0 $, our only choice is $ D $ and it clearly preserves the conditions.

Assume that $ \alpha=\beta+1 $. Then $ M_{\beta+1} $ is an elementary submodel  with the chain property
such that $ \left|M_{\beta+1}\right|<\kappa $, thus by the induction hypothesis for every digraph $ D'\in 
M_{\beta+1} $ the 
reversion sequence $ \mathcal{C}_{M_{\beta+1},D'} $ exists.
  Then $ D_{\beta+1}:=D_\beta\circlearrowleft \mathcal{C}_{M_{\beta+1},D_\beta} $  is definable from  
  $ M_{\beta+1},D_\beta\in M_{\beta+2} $ and hence $D_{\beta+1}\in M_{\beta+2} $ as demanded by ii. The preservation 
  of the other conditions follows easily from the properties of  $ \mathcal{C}_{M_{\beta+1},D_\beta} $. 
  
  Finally, let  $ \alpha $ be a limit ordinal. We first take  a generalized limit $ L_\alpha $ of $ \left\langle D_\beta: \beta<\alpha  
 \right\rangle $. 
 Recall  that, generalized limits preserve local reachability (see Proposition \ref{lim loc reach}). Then we apply Lemma 
 \ref{loc reach scat down} with 
$ D, M_\alpha, L_\alpha $  to obtain 
$D_\alpha:=D^{*}(M_\alpha, L_\alpha) $. We have $ D_\alpha\in M_{\alpha +1} $ since it is definable from   $D, 
\left\langle 
M_\beta: 
\beta\leq\alpha  \right\rangle\in M_{\alpha+1}$. Conditions iii and iv follow from  Lemma 
 \ref{loc reach scat down}.  Observe that   Lemma \ref{loc reach scat down} is applicable for 
  the triple $ D_\beta, M_\alpha, L_\alpha $ for every $ \beta<\alpha $    and gives 
 the same $ D_\alpha $. Thus  Lemma \ref{loc reach scat down} guarantees that $ D_\alpha $ is locally reachable from every $ 
 D_\beta $ for $ \beta<\alpha $.  The definition of the sequence $ \left\langle D_\alpha: \alpha\leq \kappa \right\rangle  $ is 
 complete.
 
\medspace 

We will construct $ \mathcal{C}_{M,D} $ in such a way that $ D \circlearrowleft \mathcal{C}_{M,D}=D_\kappa $ holds. 
To 
do so,  we first partition $D$ and $ D_\kappa $ into smaller pieces. Let $ D^{\alpha}:=D\cap (M_{\alpha+1}\setminus 
M_\alpha) $  and $ 
 D_\kappa^{\alpha}:= D_\kappa\cap (M_{\alpha+1}\setminus 
 M_\alpha) $. It is enough to find for all $ \alpha<\kappa $ an $ \mathcal{E}_\alpha \in \mathbb{RS}(D^{\alpha})  $ with
  $ D^{\alpha}\circlearrowleft \mathcal{E}_\alpha=D_\kappa^{\alpha} $, since then the concatenation of $ 
  \mathcal{E}_\alpha$ for $ \alpha<\kappa $ is suitable for $ \mathcal{C}_{M,D} $. By induction $ 
  \mathcal{C}_{M_{\alpha+1}, D_\alpha, D_\kappa } $ exists for every $ \alpha<\kappa $.
\begin{claim}
For every $ \alpha<\kappa $, $ (D_\alpha \circlearrowleft  \mathcal{C}_{M_{\alpha+1}, D_\alpha, D_\kappa })\cap 
M_{\alpha+1}=D_\kappa\cap 
M_{\alpha+1}$.
\end{claim}
\begin{proof}
Since $ D_\kappa $ is locally  reachable from 
 $  D_{\alpha+1}=D_\alpha \circlearrowleft  \mathcal{C}_{M_{\alpha+1}, D_\alpha }  $,  it follows directly from property 2.
\end{proof}

This means that $ \mathcal{C}_{M_{\alpha+1}, D_\alpha, D_\kappa} $ simultaneously transforms $ D^{\beta} $ to $ 
D^{\beta}_\kappa  $  for every $ \beta 
\leq\alpha $ instead of doing this only for  $ \alpha $ and  not touching $ D^{\beta} $ for $ \beta<\alpha $ as required for 
$ \mathcal{E}_\alpha  $. We fix this  by applying Corollary
 \ref{avoid down} with 
$\mathcal{C}_{M_{\alpha+1}, D_\alpha,D_\kappa},\  M_\alpha,\   D,\ D_\alpha $. We have just constructed the desired $ 
\mathcal{C}_{M,D} $. 
 
 Suppose that $ R $ is locally reachable from $ D_\kappa $. We build $ \mathcal{C}_{M,D,R} $ using similar 
 techniques as for  $ \mathcal{C}_{M,D} $. By induction $\mathcal{C}_{M_{\alpha+1}, D_\alpha, R }$ exists for 
 $\alpha<\kappa$.
 
 \begin{claim}
 For every $ \alpha<\kappa $, $ (D_\alpha \circlearrowleft  \mathcal{C}_{M_{\alpha+1}, D_\alpha, R })\cap 
 M_{\alpha+1}=R\cap 
 M_{\alpha+1}$
 \end{claim}
 \begin{proof}
 $ R$ is locally  reachable from 
  $ D_\alpha \circlearrowleft  \mathcal{C}_{M_{\alpha+1}, D_\alpha }= D_{\alpha+1}  $ by the transitivity of local reachability 
  (see Corollary \ref{loc reach trans})  since we assumed that $ R$ is 
  locally  reachable from $ D_\kappa $ which itself is locally reachable from $ D_{\alpha+1} $. In the light of this the 
  statement  follows  from 
  property 2.
 \end{proof}
 
Thus $ \mathcal{C}_{M_{\alpha+1}, D_\alpha, D_\kappa} $ transforms  $ D^{\beta} $ to $ R^{\beta}:=R\cap 
(M_{\beta+1}\setminus M_\beta)  $  
for every $ \beta  \leq\alpha $.
We construct a $ \mathcal{E}_\alpha^{*} \in \mathbb{RS}(D^{\alpha})  $  by applying Corollary \ref{avoid down} with
$ \mathcal{C}_{M_{\alpha+1}, D_\alpha, R },\  M_\alpha,\  D,\ D_\alpha  $ for which 
 $ D^{\alpha}\circlearrowleft \mathcal{E}_\alpha^{*}=R^{\alpha} $. Let $ 
 \mathcal{C}_{M,D,R} $ be the concatenation of 
  the sequences $ \mathcal{E}_\alpha^{*}\  (\alpha<\kappa)$.
\end{proof}
\section{On the structure of scattered strong digraphs}
In a scattered tournament every strong component is finite (see Theorem \ref{strong 
finite}). How simple is the structure of a scattered digraph in general?  A simple way  to construct arbitrary large scattered 
strong digraphs is glueing together 
finite strong digraphs in a tree-like structure where the neighbouring finite digraphs share exactly one vertex. One may hope that 
under the 
extra
assumption 
that the underlying undirected graph is $ 2 $-vertex-connected a strong scattered digraph must be finite, however, it is false. 
It is not too hard to prove that if every  edge of a digraph $ D $ is  in only finitely many directed cycles, then $ D $ 
is scattered.  Using this  observation, it is easy to draw a counterexample (see Figure \ref{fig example}).

\begin{figure}[H]
 \centering
\begin{tikzpicture}

\node [draw,circle] (v1) at (-3,-0.5) {};
\node [draw,circle] (v2) at (-2,-0.5) {};
\node [draw,circle] (v3) at (-1,-0.5) {};
\node [draw,circle] (v4) at (0,-0.5) {};
\node [draw,circle] (v5) at (1,-0.5) {};
\node [draw,circle] (v6) at (2,-0.5) {};

\draw [ ->] (v1) edge (v2);
\draw [ ->] (v2) edge (v3);
\draw [ ->] (v3) edge (v4);
\draw [ ->] (v4) edge (v5);
\draw [ ->] (v5) edge (v6);
\node (v7) at (3,-0.5) {};
\draw [ ->] (v6) edge (v7);

\node (v8) at (3.1,-0.5) {$\dots$};

\draw [ ->] (v3) edge[out=135,in=45] (v1);
\draw [ ->] (v4) edge[out=135,in=45]  (v2);
\draw [ ->] (v5) edge[out=135,in=45]  (v3);
\draw [ ->] (v6) edge[out=135,in=45]  (v4);

\draw [ ->] (v8) edge[out=135,in=45]  (v5);
\end{tikzpicture}
 \caption{An example for an infinite strong digraph $ D  $ where the underlying undirected graph is $ 2 $-vertex-connected 
 and $ D $ is 
 scattered.}\label{fig example}
 \end{figure}
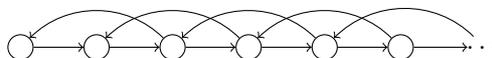

 It is natural to ask if there are  uncountable strong scattered digraphs where the underlying undirected graph is $ 2 
 $-vertex-connected. To show  the negative answer, we use only the weaker 
 assumption  that  the local 
  edge-connectivities are countable  instead of being scattered.   (Strong 
and scattered implies  that  the local edge-connectivities are actually
finite, see Corollary \ref{separate two}.)

\begin{prop}\label{transform to countable}
If $ D $ is a strongly connected  digraph in which every local edge-connectivity is countable   and  $ \mathsf{Un}(D) $ is 
2-vertex-connected, 
then $ D 
$ is countable.
\end{prop}
\begin{proof}
Pick a countable elementary 
submodel $ M\ni D $ and for $v\in V(D)\cap M$ let $K_v$ be the strong component of $D\setminus M $ containing $v$ if $ v\in 
V(D\setminus M) $ and let $ \{ v \} $ otherwise. We 
cannot have $K_u=K_v$ for $u\neq v$ since otherwise we would have an uncountable local edge-connectivity by Fact \ref{big 
loc connect}, contradicting the assumption. The same argument shows that there are no edges between the sets 
$K_v$ in $D\setminus M $. Suppose for a contradiction that there exists a $w\in V(D)$ which is not in  any $K_v$.
Since $ D 
$ is strongly connected, there is a path $P$ from $ V(D)\cap M $ to $w$  and a path $Q$ from $w$ to $ V(D)\cap M $ in 
$D\setminus M $. Let us denote the first vertex of $P$ by $u$ and the last vertex of $Q$ by $v$. By the choice of $w$ we have 
$u\neq v$. But then the concatenation of $P$ and $Q$ shows that $v$ is reachable from $u$ in  $D\setminus M $ contradicting 
Fact \ref{big loc connect}.

It follows that the set of the strong components of $D\setminus M $ are exactly $\{K_v: v\in V(D\setminus M)\cap M\} $, 
furthermore, each 
strong component is a weak component as well. Therefore if an edge leaves $ K_v $ or arrives to $ K_v $ in $ D $ it must be 
incident with $ v $. 
Since $ \mathsf{Un}(D) $ 
is $ 2 $-vertex-connected, it implies $ K_v=\{ v \} $ for all $ v\in V(D)\cap M $ and hence $ D\subseteq M $.
\end{proof}

\section{Open problems}
Proposition \ref{fin many cycle} ensures that in  the finite case (Theorem \ref{finite thm}) one can actually reverse pairwise 
edge-disjoint cycles to reduce the 
dichromatic number to at most two. In our proof of the conjecture we can guarantee by using Corollary \ref{touch at most 2} 
that for every edge $ e $ 
there are at most two cycles in the sequence that contain some orientation of $ e $. It definitely cannot be improved to ``at most 
one'' in 
Theorems \ref{max scat from everybody} and \ref{strong finite}. Indeed, reversing
pairwise edge-disjoint cycles cannot change the  local edge-connectivities. We do not know (not even for countable digraphs) if it 
is always possible to reduce the dichromatic number 
to at most two by reversing pairwise edge-disjoint cycles.
\begin{ques}
Is it possible to find in an arbitrary digraph $ D $ a family of pairwise edge-disjoint directed cycles in such a way that if we 
reverse their united edge set then the resulting digraph $ D^{*} $ has dichromatic number at most two?
\end{ques}

Yet another interesting line of research would be to allow the use of $\mathbb Z$-chains (two-way infinite paths) in both the 
definition of the dichromatic number and during the cycle reversions. Let us call these notions the generalized dichromatic 
number and generalized cycle reversion.

\begin{ques}[R. Diestel] Assume $D$ is an arbitrary digraph. Is there a generalized cycle reversion which lowers the generalized 
dichromatic number of $D$ to at most $2$?
\end{ques}

The latter is open even for (countably infinite) tournaments and we also wonder if one can exploit some connections to the theory 
of ends and tangles.

\section{Appendix on elementary submodels}\label{appendix}
Roughly speaking, an elementary submodel of a larger model $\mathbb V$, where in our case $\mathbb V$ is always the whole 
set-theoretic universe, is a  an $M \subseteq\mathbb V$ such that the structure $ (M,\in) $ is a very close approximation of 
$\mathbb V$. Note that $M$ might be countable while $\mathbb V$ is a proper class.

By approximation, we mean that whenever a first-order statement holds in the universe $\mathbb V$ about some elements of $ M 
$ then it is already true inside the structure $ (M,\in) $. The elementary 
submodel technique
 is an  
efficient tool in combinatorial set theory and in several other fields. Let us summarize the basic facts that we use in the paper.  For 
an excellent 
detailed
introduction for elementary submodels we suggest \cite{ElementaryBasic} and for more advanced applications one can see 
\cite{ElementaryAdvanced}. 

First of all, let us make the definition precise. For $ k<n $, let $ \varphi_k$ be a formula in the language of set theory with $ n_k $ 
free 
variables and let $\Sigma = \{ \varphi_k: k<n\}$.  
A $ \Sigma$-\textbf{elementary submodel} is a set $ M $ such that
\[ \bigwedge_{k<n}\left[ (\forall x_1,\dots, x_{n_k}\in M) (\varphi_k(x_1,\dots, x_{n_k}) \Longleftrightarrow (M,\in)\models 
\varphi_k(x_1,\dots, 
x_{n_k}) )  \right]. 
\]

In plain words, we say that the formulas $\varphi$ are \textit{absolute} between $\mathbb V$ and $M$. In this paper, we need 
only the following basic facts about elementary submodels.

First, for any infinite  cardinal  one can  construct an elementary submodel containing a given set.

\begin{fact}\label{exist submod}
For every finite set $ \Sigma $ of formulas, infinite cardinal $ \kappa $ and set $ X $ there is a $ \Sigma $-elementary submodel $ 
M 
$ 
such that $ \kappa \cup \{X \}\subseteq M $ and $M$ has size $\kappa$.
\end{fact}

Second, $M$ has very strong closure properties: if a set $Y$ is uniquely definable by a formula using parameters from an 
elementary submodel $M$ then $Y$ must be in $M$ already.

\begin{fact}\label{contains elem}
Let $ M $ be a $ \Sigma $-elementary submodel and be $Y$ any set. 
\begin{enumerate}
    \item Suppose that there is some $\varphi\in \Sigma$ and $p_i\in M$ so that $\varphi(p_1,\dots, p_n,y)$ only holds for $y=Y$. 
    Then $Y\in M$.
    \item Suppose that $Y\in M$ and the cardinality $\kappa$ of $Y$ is a subset of $M$. Then $Y\subset M$ as well assuming that 
    $\Sigma$ is large enough.\footnote{$ \Sigma $ contains formulas like ``$ \kappa $ is the cardinality of $ Y $'', ``$ f $ is a 
    bijection from $ \kappa$ to $ Y $`` etc.}
\end{enumerate}
\end{fact}

The second statement is somewhat subtle if one encounters elementary submodels for the first time. By the first fact, we can find 
countable elementary submodels $M$ which contain the real line $\mathbb R$ as an element (i.e., $X=\{\mathbb R\}$). 
However, $\mathbb R$ cannot be a subset of $M$ simply by cardinality.

Also, it is common practice when using elementary submodels  to omit $ \Sigma $ completely  from the discussion. In essentially 
all cases, it is clear (but tedious to write out)
 for which finitely many
formulas the absoluteness arguments are used. 
\medskip

Finally, we make use of the following fact regularly.

\begin{fact}\label{big loc connect}
Let $ D $ be a digraph and let $ M\ni D $ be an elementary submodel with $ \left|M\right|=\kappa\subseteq M $.  If $u,v\in M\cap 
V(D)$ are distinct and $ 
\lambda(u,v;D\setminus M)>0 $ then 
$ 
\lambda(u,v;D)>\kappa $ and $  \lambda(u,v;D\cap M)=\kappa$.
\end{fact} 
\begin{proof} First, let us prove $ 
\lambda(u,v;D)>\kappa$. Assume on the contrary, that any pairwise edge-disjoint system of paths from $u$ to $v$ has size at 
most $\kappa$. Take a maximal such family $\mathcal P$, which is also an element of $M$ (this can be done by absoluteness). 
Now, by applying the previous fact twice, all the paths $P\in \mathcal P$ must be elements of $M$ and again for these finite 
paths, we see that all the edges in these paths $P$ are in $M$ too. Now, the assumption  $ 
\lambda(u,v;D\setminus M)>0 $ says that there is a path $Q$ from $u$ to $v$ that has no edges in $M$. However, then $\mathcal 
P\cup \{Q\}$ is a strictly larger edge-disjoint family violating the maximality of $\mathcal P$.

Second, since in $D$, there is a family $\mathcal P$ of pairwise edge-disjoint paths from $u$ to $v$ of size $\kappa$ (even one 
of size $\kappa^+$), there must exist such a family $\mathcal P'$ in $M$. Now, apply the previous fact twice as before to see that 
all the paths $P\in \mathcal P'$ are in $M$ and  all the edges in these paths $P$ are in $M$ too. So,  $  \lambda(u,v;D\cap M)\geq 
\kappa$ must hold and therefore $  \lambda(u,v;D\cap M)=\kappa$ since $M\cap D$ has size $\kappa$.

\end{proof}

\end{document}